\documentclass[11pt, reqno]{amsart}
\usepackage{amsfonts}
\usepackage{amsmath}
\usepackage{amssymb}
\usepackage{amscd}
\usepackage[mathscr]{eucal}
\usepackage{url}

\allowdisplaybreaks[1]

\renewcommand*{\bar}{\overline}
\newcommand{\gfp}[1]{\Gamma_p{\left({#1}\right)}}
\newcommand{\biggfp}[1]{\Gamma_p{\bigl({#1}\bigr)}}

\theoremstyle{plain}

\newtheorem{theorem}{Theorem}[section]
\newtheorem{lemma}[theorem]{Lemma}
\newtheorem{prop}[theorem]{Proposition}
\newtheorem{cor}[theorem]{Corollary}
\theoremstyle{definition}
\newtheorem{defi}[theorem]{Definition}
\newtheorem{remark}[theorem]{Remark}

\numberwithin{equation}{section}

\begin{document}

\title[Trace of Frobenius and the $p$-adic gamma function]{The trace of Frobenius of elliptic curves\\ and the $p$-adic gamma function}
\author{Dermot M\lowercase{c}Carthy}  

\address{Department of Mathematics, Texas A\&M University, College Station, TX 77843-3368, USA}

\email{mccarthy@math.tamu.edu}




\begin{abstract}
We define a function in terms of quotients of the $p$-adic gamma function which generalizes earlier work of the author on extending hypergeometric functions over finite fields to the $p$-adic setting.
We prove, for primes $p > 3$, that the trace of Frobenius of any elliptic curve over $\mathbb{F}_p$, whose $j$-invariant does not equal $0$ or $1728$, is just a special value of this function.
This generalizes results of Fuselier and Lennon which evaluate the trace of Frobenius in terms of hypergeometric functions over $\mathbb{F}_p$ when $p \equiv 1 \pmod {12}$.
\end{abstract}

\maketitle


\section{Introduction and Statement of Results}\label{sec_Intro}
Let $\mathbb{F}_p$ denote the finite field with $p$, a prime, elements. Consider $E/\mathbb{Q}$ an elliptic curve with an integral model of discriminant $\Delta(E)$.
We denote $E_p$ the reduction of $E$ modulo $p$. We note that $E_p$ is non-singular, and hence an elliptic curve over $\mathbb{F}_p$, if and only if $p \nmid \Delta(E)$, in which case we say $p$ is a prime of good reduction.
Regardless, we define
\begin{equation}\label{def_ap}
a_p(E) := p+1-\# E_p(\mathbb{F}_p).
\end{equation}
If $p$ is not a prime of good reduction we know $a_p(E) =0,\pm1$ depending on the nature of the singularity.
If $p$ is a prime of good reduction, we refer to $a_p(E)$ as the \emph{trace of Frobenius} as it can be interpreted as the trace of the Frobenius endomorphism of $E/\mathbb{F}_p$.
For a given elliptic curve $E/\mathbb{Q}$, these $a_p$ are important quantities. Recall the Hasse-Weil $L$-function of $E$ (viewed as function of a complex variable $s$) is defined by
$$L(E,s):=\prod_{p \mid \Delta} \frac{1}{1-a_p(E) p^{-s}} \prod_{p \nmid \Delta} \frac{1}{1-a_p(E) p^{-s}+p^{1-2s}}.$$
This Euler product converges for Re$(s)>\frac{3}{2}$ and has analytic continuation to the whole complex plane. 
The Birch and Swinnerton-Dyer conjecture concerns the behavior of $L(E,s)$ at $s=1$.

The main result of this paper relates the trace of Frobenius to a special value of a function which we define in terms of quotients of the $p$-adic gamma function.
Let $\gfp{\cdot}$ denote Morita's $p$-adic gamma function and let $\omega$ denote the Teichm\"{u}ller character of $\mathbb{F}_p$ with $\bar{\omega}$ denoting its character inverse. 
For $x \in \mathbb{Q}$ we let  $\left\lfloor x \right\rfloor$ denote the greatest integer less than or equal to $x$ and
$\langle x \rangle$ denote the fractional part of $x$, i.e. $x- \left\lfloor x \right\rfloor$.
\begin{defi}\label{def_Gp}
Let $p$ be an odd prime and let $t \in \mathbb{F}_p$. For $n \in \mathbb{Z}^{+}$ and $1 \leq i \leq n$, let $a_i, b_i \in \mathbb{Q} \cap \mathbb{Z}_p$.
Then we define  
\begin{multline*}\label{for_GFn}
{_{n}G_{n}}
\biggl[ \begin{array}{cccc} a_1, & a_2, & \dotsc, & a_n \\
 b_1, & b_2, & \dotsc, & b_n \end{array}
\Big| \; t \; \biggr]_p
: = \frac{-1}{p-1}  \sum_{j=0}^{p-2} 
(-1)^{jn}\;
\bar{\omega}^j(t)\\
\times \prod_{i=1}^{n} 
\frac{\biggfp{\langle a_i -\frac{j}{p-1}\rangle}}{\biggfp{\langle a_i \rangle}}
\frac{\biggfp{\langle - b_i +\frac{j}{p-1}\rangle}}{\biggfp{\langle -b_i \rangle}}
(-p)^{-\lfloor{\langle a_i \rangle -\frac{j}{p-1}}\rfloor -\lfloor{\langle -b_i \rangle +\frac{j}{p-1}}\rfloor}.
\end{multline*}
\end{defi}
\noindent Throughout the paper we will refer to this function as ${_{n}G_{n}}[\cdots]$.
The value of ${_{n}G_{n}}[\cdots]$ depends only on the fractional part of the $a$ and $b$ parameters. Therefore, we can assume $0 \leq a_i, b_i <1$.

This function has some very nice properties. 
It generalizes the function defined by the author in \cite{McC4} which exhibits relationships to Fourier coefficients of modular forms. 
This earlier function has only one line of parameters and corresponds to ${_{n}G_{n}}[\cdots]$ when all the bottom line parameters are integral and $t=1$. 
The earlier function also extended, to the $p$-adic setting, hypergeometric functions over finite fields with trivial bottom line parameters. 
In Section \ref{sec_Properties} we will see that ${_{n}G_{n}}[\cdots]$ extends hypergeometric functions over finite fields in their full generality, to the $p$-adic setting. 
By definition, results involving  hypergeometric functions over finite fields will often be restricted to primes in certain congruence classes (see for example \cite{E, F2, L, M, V}). 
The motivation for developing ${_{n}G_{n}}[\cdots]$ is that it can often allow these results to be extended to a wider class of primes \cite{McC4, McC5}, as we exhibit in our main result below. 
We will discuss these properties in more detail in Section \ref{sec_Properties}.

We now state our main result which relates the trace of Frobenius of an elliptic curve over $\mathbb{F}_p$ to a special value of ${_{n}G_{n}}[\cdots]$. 
We first note that if $p>3$ then any elliptic curve over $\mathbb{F}_p$ is isomorphic to an elliptic curve of the form $E:y^2 = x^3+ax+b$, i.e., short Weierstrass form,
and that the trace of Frobenius of isomorphic curves are equal. Let $j(E)$ denote the $j$-invariant of the elliptic curve $E$. 
Let $\phi_p(\cdot)$ be the Legendre symbol modulo $p$.  We will often omit the subscript $p$ when it is clear from the context.
\begin{theorem}\label{thm_trace}
Let $p>3$ be prime. Consider an elliptic curve $E/\mathbb{F}_p$ of the form $E:y^2 = x^3+ax+b$ with $j(E)\neq 0, 1728$. Then
\begin{equation}\label{for_main}
a_p(E) = \phi(b) \cdot p \cdot
{_{2}G_{2}}
\biggl[ \begin{array}{cc} \frac{1}{4},  & \frac{3}{4}\vspace{.05in}\\
 \frac{1}{3},  & \frac{2}{3} \end{array}
\Big| \; {-\frac{27b^2}{4a^3}} \; \biggr]_p. 
\end{equation}
\end{theorem}
\noindent Independent of Theorem \ref{thm_trace}, we will see later from Proposition \ref{prop_delta} that the right-hand side of (\ref{for_main}) is $p$-integral. 

Theorem \ref{thm_trace} generalizes results of Fuselier \cite[Theorem 1.2]{F2} and Lennon \cite[Theorem 2.1]{L} which evaluate the trace of Frobenius in terms of hypergeometric functions over $\mathbb{F}_p$ when $p \equiv 1 \pmod {12}$.
The results in \cite{L} are in fact over $\mathbb{F}_q$, for $q \equiv 1 \pmod {12}$ a prime power, and hence allow calculation of $a_p$ up to sign when $p \not\equiv 1 \pmod {12}$ via the relation $a_p^2 = a_{p^2} +2p$. 
Theorem \ref{thm_trace} however gives a direct evaluation of $a_p$ for all primes $p>3$ and resolves this sign issue.

One of the nice features of the main result in \cite{L} is that it is independent of the Weierstrass model of the elliptic curve. Recall an elliptic curve over a field $\mathbb{K}$ in Weierstrass form is given by
\begin{equation}\label{for_WF}
E: y^2 +a_1xy +a_3y=x^3+a_2x^2 +a_4x+a_6,
\end{equation}
with $a_1, a_2, \dots, a_6 \in \mathbb{K}$.
We can define the quantities 
$b_2 := {a_1}^2 +4 a_2;$
$b_4 := 2 a_4 + a_1a_3;$
$b_6 := {a_3}^2 + 4 a_6;$
$b_8 := {a_1}^2 a_6 + 4 a_2 a_6 - a_1 a_3 a_4 + a_2 {a_3}^2 - {a_4}^2;$
$c_4 := b_2^2 - 24b_4;$ and
$c_6 := -b_2^3+36b_2 b_4 -216 b_6,$
in the standard way. These can then be used to calculate 
$\Delta(E)=\frac{c_4^3-c_6^2}{1728}$ and $j(E)=\frac{c_4^3}{\Delta(E)}$.
An admissible change of variables,  $x=u^2 x^{\prime} +r $ and $y= u^3 y^{\prime} + s u^2 x^{\prime} +t$ with $u,r,s,t, \in \mathbb{K}$ and $u \neq 0$, in (\ref{for_WF}) will result in an isomorphic curve also given in Weierstrass form, 
and any two isomorphic curves over $\mathbb{K}$ are related by such an admissible change of variables. 
Two curves related by an admissible change of variables will have the same $j$-invariant but their discriminants will differ by a factor of a twelfth-power, namely $u^{12}$ and their respective $c_i$ quantities will differ by a factor of $u^i$.
This allows the main result in \cite{L}, which is stated in terms of $j(E)$ and $\Delta(E)$, to be expressed independently of the Weierstrass model of the elliptic curve.
We can do something similar with Theorem \ref{thm_trace}.
\begin{cor}\label{cor_trace}
Let $p>3$ be prime. Consider an elliptic curve $E/\mathbb{F}_p$ in Weierstrass form with $j(E)\neq 0, 1728$. Then
$$a_p(E) = \phi(-6\cdot c_6) \cdot p \cdot
{_{2}G_{2}}
\biggl[ \begin{array}{cc} \frac{1}{4},  & \frac{3}{4}\vspace{.05in}\\
 \frac{1}{3},  & \frac{2}{3} \end{array}
\Big| \; 1- \frac{1728}{j(E)} \; \biggr]_p.$$
\end{cor}

Please refer to \cite{Kn, Si} for a detailed account of any of the properties of elliptic curves mentioned in the above discussion.
The rest of this paper is organized as follows. In Section \ref{sec_Prelim} we recall some basic properties of multiplicative characters, Gauss sums and the $p$-adic gamma function.
We discuss some properties of ${_{n}G_{n}}[\cdots]$ in Section \ref{sec_Properties} including its relationship to hypergeometric functions over finite fields. The proofs of our main results are contained in Section \ref{sec_Proofs}.
Finally, we make some closing remarks in Section \ref{sec_cr}.


\section{Preliminaries}\label{sec_Prelim}
Let $\mathbb{Z}_p$ denote the ring of $p$-adic integers, $\mathbb{Q}_p$ the field of $p$-adic numbers, $\bar{\mathbb{Q}_p}$ the algebraic closure of $\mathbb{Q}_p$, and $\mathbb{C}_p$ the completion of $\bar{\mathbb{Q}_p}$.
\subsection{Multiplicative Characters and Gauss Sums}
Let $\widehat{\mathbb{F}^{*}_{p}}$ denote the group of multiplicative characters of $\mathbb{F}^{*}_{p}$. 
We extend the domain of $\chi \in \widehat{\mathbb{F}^{*}_{p}}$ to $\mathbb{F}_{p}$, by defining $\chi(0):=0$ (including the trivial character $\varepsilon$) and denote $\bar{\chi}$ as the inverse of $\chi$. 
We recall the following orthogonal relations. 
For $\chi \in \widehat{\mathbb{F}_p^{*}}$ we have
\begin{equation}\label{for_TOrthEl}
\sum_{x \in \mathbb{F}_p} \chi(x)=
\begin{cases}
p-1 & \text{if $\chi = \varepsilon$}  ,\\
0 & \text{if $\chi \neq \varepsilon$}  ,
\end{cases}
\end{equation}
and, for $x \in \mathbb{F}_p$ we have
\begin{equation}\label{for_TOrthCh}
\sum_{\chi \in \widehat{\mathbb{F}_p^{*}}} \chi(x)=
\begin{cases}
p-1 & \text{if $x=1$}  ,\\
0 & \text{if $x \neq 1$}  .
\end{cases}
\end{equation}

\noindent We now introduce some properties of Gauss sums. For further details see \cite{BEW}, noting that we have adjusted results to take into account $\varepsilon(0)=0$. 
Let $\zeta_p$ be a fixed primitive $p$-th root of unity in $\bar{\mathbb{Q}_p}$. We define the additive character $\theta : \mathbb{F}_p \rightarrow \mathbb{Q}_p(\zeta_p)$ by $\theta(x):=\zeta_p^{x}$. 
\noindent 
It is easy to see that
\begin{equation}\label{for_AddProp}
\theta(a+b)=\theta(a) \theta(b).
\end{equation}
and
\begin{equation}\label{sum_AddChar}
\sum_{x \in \mathbb{F}_p} \theta(x) = 0.
\end{equation}
We note that $\mathbb{Q}_p$ contains all $(p-1)$-th roots of unity and in fact they are all in $\mathbb{Z}^{*}_p$. Thus we can consider multiplicative characters of $\mathbb{F}_p^{*}$ to be maps $\chi: \mathbb{F}_p^{*} \to \mathbb{Z}_{p}^{*}$. 
Recall then that for $\chi \in \widehat{\mathbb{F}_p^{*}}$, the Gauss sum $g(\chi)$ is defined by 
$g(\chi):= \sum_{x \in \mathbb{F}_p} \chi(x) \theta(x).$
It easily follows from (\ref{for_TOrthCh}) that we can express the additive character as a sum of Gauss sums. Specifically, for $x \in \mathbb{F}_p^*$\hspace{1pt} we have
\begin{equation}\label{for_AddtoGauss}
\theta(x)= \frac{1}{p-1}\sum_{\chi \in \widehat{\mathbb{F}_p^{*}}} g(\bar{\chi})  \, \chi(x).
\end{equation}
The following important result gives a simple expression for the product of two Gauss sums. For $\chi \in \widehat{\mathbb{F}_p^{*}}$ we have
\begin{equation}\label{for_GaussConj}
g(\chi)g(\bar{\chi})=
\begin{cases}
\chi(-1) p & \text{if } \chi \neq \varepsilon,\\
1 & \text{if } \chi= \varepsilon.
\end{cases}
\end{equation}
\noindent Another important product formula for Gauss sums is the Hasse-Davenport formula.
\begin{theorem}[Hasse, Davenport \cite{BEW} Thm 11.3.5]\label{thm_HD}
Let $\chi$ be a character of order $m$ of $\mathbb{F}_p^*$ for some positive integer $m$. For a character $\psi$ of $\mathbb{F}_p^*$ we have
\begin{equation*}
\prod_{i=0}^{m-1} g(\chi^i \psi) = g(\psi^m) \psi^{-m}(m)\prod_{i=1}^{m-1} g(\chi^i).
\end{equation*}
\end{theorem}

We now recall a formula for counting zeros of polynomials in affine space using the additive character.
If $f(x_1, x_2, \ldots, x_n) \in \mathbb{F}_p[x_1, x_2, \ldots, x_n]$, then the number of points, $N_p$, in $\mathbb{A}^n(\mathbb{F}_p)$ satisfying
$f(x_1, x_2, \ldots, x_n) =0$ is given by
\begin{equation}\label{for_CtgPts}
p N_p = p^n +\sum_{y \in \mathbb{F}_p^*} \sum_{x_1, x_2, \ldots, x_n \in \mathbb{F}_p}
\theta(y \: f(x_1, x_2, \ldots, x_n)) \; .
\end{equation}

\subsection{$p$-adic preliminaries}\label{subsec_padicPrelim}
We define the Teichm\"{u}ller character to be the primitive character $\omega: \mathbb{F}_p \rightarrow\mathbb{Z}^{*}_p$ satisfying $\omega(x) \equiv x \pmod p$ for all $x \in \{0,1, \ldots, p-1\}$.
We now recall the $p$-adic gamma function. For further details, see \cite{Ko}.
Let $p$ be an odd prime.  For $n \in \mathbb{Z}^{+}$ we define the $p$-adic gamma function as
\begin{align*}
\gfp{n} &:= {(-1)}^n \prod_{\substack{0<j<n\\p \nmid j}} j \\
\intertext{and extend to all $x \in\mathbb{Z}_p$ by setting $\gfp{0}:=1$ and} 
\gfp{x} &:= \lim_{n \rightarrow x} \gfp{n}
\end{align*}
for $x\neq 0$, where $n$ runs through any sequence of positive integers $p$-adically approaching $x$. 
This limit exists, is independent of how $n$ approaches $x$, and determines a continuous function
on $\mathbb{Z}_p$ with values in $\mathbb{Z}^{*}_p$.
We now state a product formula for the $p$-adic gamma function.
If $m\in\mathbb{Z}^{+}$, $p \nmid m$ and $x=\frac{r}{p-1}$ with $0\leq r \leq p-1$ then
\begin{equation}\label{for_pGammaMult}
\prod_{h=0}^{m-1} \gfp{\tfrac{x+h}{m}}=\omega\left(m^{(1-x)(1-p)}\right)
\gfp{x} \prod_{h=1}^{m-1} \gfp{\tfrac{h}{m}}.
\end{equation}
We note also that
\begin{equation}\label{for_pGammaOneMinus}
\gfp{x}\gfp{1-x} = {(-1)}^{x_0},
\end{equation}
where $x_0 \in \{1,2, \dotsc, {p}\}$ satisfies $x_0 \equiv x \pmod {p}$.
\noindent The Gross-Koblitz formula \cite{GK} allows us to relate Gauss sums and the $p$-adic gamma function. Let $\pi \in \mathbb{C}_p$ be the fixed root of $x^{p-1}+p=0$ which satisfies ${\pi \equiv \zeta_p-1 \pmod{{(\zeta_p-1)}^2}}$. Then we have the following result.
\begin{theorem}[Gross, Koblitz \cite{GK}]\label{thm_GrossKoblitz}
For $ j \in \mathbb{Z}$,
\begin{equation*} 
g(\bar{\omega}^j)=-\pi^{(p-1) \langle{\frac{j}{p-1}}\rangle} \: \gfp{\langle{\tfrac{j}{p-1}}\rangle}.
\end{equation*}
\end{theorem}


\section{Properties of ${_{n}G_{n}}[\cdots]$.}\label{sec_Properties}
As both $\gfp{\cdot}$ and $\omega(\cdot)$ are in $\mathbb{Z}^{*}_p$, we see immediately from its definition that ${_{n}G_{n}}[\cdots]_p \in p^{\delta}\mathbb{Z}_p$ for some $\delta \in \mathbb{Z}$. 
We describe $\delta$ explicitly in the following proposition.
We first define 
$$
\langle b_i \rangle ^{*} := 1- \langle -b_i \rangle = 
\begin{cases}
\langle b_i \rangle & \textup{if }  b_i \notin \mathbb{Z},\\
1 &  \textup{if }  b_i  \in \mathbb{Z}.
\end{cases}
$$
\begin{prop}\label{prop_delta}
Let $p$ be an odd prime and let $t \in \mathbb{F}_p$. Let $n \in \mathbb{Z}^{+}, 1 \leq i \leq n$ and $a_i, b_i \in \mathbb{Q} \cap \mathbb{Z}_p$. 
For $j \in \mathbb{Z}$ we define
$$f(j):= \#\{a_i \mid \langle a_i \rangle  < \tfrac{j}{p-1}, 1 \leq i \leq n\} - \#\{b_i \mid \langle b_i \rangle^{*} \leq \tfrac{j}{p-1}, 1 \leq i \leq n\}.$$
Then
$${_{n}G_{n}}
\biggl[ \begin{array}{cccc} a_1, & a_2, & \dotsc, & a_n \\
 b_1, & b_2, & \dotsc, & b_n \end{array}
\Big| \; t \; \biggr]_p
\in p^{\delta}\mathbb{Z}_p,$$ 
where
$$\delta = \textup{Min} \{f(j) \mid 0 \leq j \leq p-2\}.$$
\end{prop}

\begin{proof}
As $\gfp{\cdot}, \omega(\cdot)$ and  $\frac{1}{p-1}$ are all in $\mathbb{Z}^{*}_p$, the result follows from noting that 
\begin{align*}
\lfloor{\langle a_i \rangle -\tfrac{j}{p-1}}\rfloor =
\begin{cases}
-1 & \textup{if }  \langle a_i \rangle < \frac{j}{p-1},\\
0 &  \textup{if }  \langle a_i \rangle \geq \frac{j}{p-1},
\end{cases}
&& \textup{and} && 
\lfloor{\langle -b_i \rangle +\tfrac{j}{p-1}}\rfloor =
\begin{cases}
1 & \textup{if }  \langle b_i \rangle ^{*}\leq \frac{j}{p-1},\\
0 &  \textup{if }  \langle b_i \rangle ^{*} > \frac{j}{p-1}.
\end{cases} 
\end{align*} 
\end{proof}

We note that ${_{n}G_{n}}[\cdots]$ generalizes the function defined by the author in \cite{McC4}. 
This earlier function has only one line of parameters and corresponds to ${_{n}G_{n}}[\cdots]$ when all the bottom line parameters are integral and $t=1$. 
Therefore the results from \cite{McC4, McC5} can be restated using ${_{n}G_{n}}[\cdots]$. 
The motivation for developing ${_{n}G_{n}}[\cdots]$ and its predecessor was to allow results involving hypergeometric functions over finite fields, which are often restricted to primes in certain congruence classes, to be extended to a wider class of primes.
While the function defined in \cite{McC4} extended, to the $p$-adic setting, hypergeometric functions over finite fields with trivial bottom line parameters, we now show, in Lemma \ref{lem_G_to_F}, that ${_{n}G_{n}}[\cdots]$ extends hypergeometric functions over finite fields in their full generality.

Hypergeometric functions over finite fields were originally defined by Greene \cite{G}, who first established these functions as analogues of classical hypergeometric functions.
Functions of this type were also introduced by Katz \cite{K} about the same time.
In the present article we use a normalized version of these functions defined by the author in \cite{McC6}, which is more suitable for our purposes.
The reader is directed to \cite[\S 2]{McC6} for the precise connections among these three classes of functions.

\begin{defi}\cite[Definition 1.4]{McC6}\label{def_F} 
For $A_0,A_1,\dotsc, A_n, B_1, \dotsc, B_n \in \widehat{\mathbb{F}_p^{*}}$ and $x \in \mathbb{F}_{p}$ define
\begin{multline}\label{def_HypFnFF}
{_{n+1}F_{n}} {\biggl( \begin{array}{cccc} A_0, & A_1, & \dotsc, & A_n \\
 \phantom{B_0,} & B_1, & \dotsc, & B_n \end{array}
\Big| \; x \biggr)}_{p}\\
:= \frac{1}{p-1}  \sum_{\chi \in \widehat{\mathbb{F}_p^{*}}} 
\prod_{i=0}^{n} \frac{g(A_i \chi)}{g(A_i)}
\prod_{j=1}^{n} \frac{g(\bar{B_j \chi})}{g(\bar{B_j})}
 g(\bar{\chi})
 \chi(-1)^{n+1}
 \chi(x).
 \end{multline}
\end{defi}
\noindent Many of the results concerning hypergeometric functions over finite fields that we quote from other articles were originally stated using Greene's function. 
If this is the case, note then that we have reformulated them in terms ${_{n+1}F_{n}}(\cdots)$ as defined above.  

We have the following relationship between ${_{n}G_{n}}[\cdots]$ and ${_{n+1}F_{n}}(\cdots)$.
\begin{lemma}\label{lem_G_to_F}
For a fixed odd prime $p$, let $A_i, B_k \in \widehat{\mathbb{F}_p^{*}}$ be given by $\bar{\omega}^{a_i(p-1)}$ and $\bar{\omega}^{b_k(p-1)}$ respectively, where $\omega$ is the Teichm\"{u}ller character . Then
\begin{equation*}
{_{n+1}F_{n}} {\biggl( \begin{array}{cccc} A_0, & A_1, & \dotsc, & A_n \\
 \phantom{B_0,} & B_1, & \dotsc, & B_n \end{array}
\Big| \; t \biggr)}_{p}
=
{_{n+1}G_{n+1}}
\biggl[ \begin{array}{cccc} a_0, & a_1, & \dotsc, & a_n \\
 0, & b_1, & \dotsc, & b_n \end{array}
\Big| \; t^{-1} \; \biggr]_p.
\end{equation*}
\end{lemma}
\begin{proof}
Starting from the definition of  ${_{n+1}F_{n}}(\cdots)$,  we convert the right-hand side of  (\ref{def_HypFnFF}) to an expression involving the $p$-adic gamma function and Teichm\"{u}ller character.
We note  $\widehat{\mathbb{F}_p^{*}}$ can be given by $\{\omega^j \mid 0 \leq j \leq p-2\}$. Then, straightforward applications of  the Gross-Koblitz formula (Theorem \ref{thm_GrossKoblitz}) with $\chi=\omega^j$ yield
$$g(\bar{\chi}) = -\pi^j \biggfp{\tfrac{j}{p-1}},$$
$$\frac{g(A_i \chi)}{g(A_i)} = \pi^{-j-(p-1) (\lfloor{a_i-\frac{j}{p-1}}\rfloor  - \lfloor{a_i}\rfloor)} \;    \frac{\biggfp{\langle a_i -\frac{j}{p-1}\rangle}}{\biggfp{\langle a_i \rangle}}$$
and
$$\frac{g(\bar{B_k \chi})}{g(\bar{B_k})} = \pi^{j-(p-1) (\lfloor{-b_k+\frac{j}{p-1}}\rfloor  - \lfloor{-b_k}\rfloor)} \;    \frac{\biggfp{\langle -b_k +\frac{j}{p-1}\rangle}}{\biggfp{\langle -b_k \rangle}},$$
where $\pi$ is as defined in Section \ref{subsec_padicPrelim}. Substituting these expressions into (\ref{def_HypFnFF}) and tidying up yields the result. 
\end{proof}

We note that if $\chi \in \widehat{\mathbb{F}_p^{*}}$ is a character of order $d$ and is given by $\bar{\omega}^{x(p-1)}$  then $x = \frac{m}{d} \in \mathbb{Q}$ and $p\equiv 1 \pmod d$.
Therefore, given a hypergeometric function over $\mathbb{F}_p$ whose arguments are characters of prescribed order, the function will only be defined for primes $p$ in certain congruence classes.
By Lemma \ref{lem_G_to_F}, for primes in these congruence classes, the finite field hypergeometric function will be related to an appropriate ${_{n}G_{n}}[\cdots]$ function.
However this corresponding ${_{n}G_{n}}[\cdots]$ will be defined at all primes not dividing the orders of the particular characters appearing in the the finite field hypergeometric function.  
This opens the possibility of extending results involving hypergeometric functions over finite fields to all but finitely many primes.

For example, we have the following result from \cite{McC5} which relates a special value of the hypergeometric function over finite fields to a $p$-th Fourier coefficient of a certain modular form.
Let 
\begin{equation}\label{for_ModForm}
f(z):= f_1(z)+5f_2(z)+20f_3(z)+25f_4(z)+25f_5(z)=\sum_{n=1}^{\infty} c(n) q^n
\end{equation}
where $f_i(z):=\eta^{5-i}(z) \hspace{2pt} \eta^4(5z) \hspace{2pt} \eta^{i-1}(25z)$,
$\eta(z):=q^{\frac{1}{24}} \prod_{n=1}^{\infty}(1-q^n)$ is the Dedekind eta function and $q:=e^{2 \pi i z}$. Then $f$ is a cusp form of weight four on the congruence subgroup $\Gamma_0(25)$.  
\begin{theorem}\cite[Corollary 1.6]{McC5}\label{Cor_GHStoMod}
If $p \equiv 1 \pmod 5$ is prime, $\chi_5 \in \widehat{\mathbb{F}^{*}_{p}}$ is a character of order $5$ and $c(p)$ is as defined in (\ref{for_ModForm}), then 
\begin{equation*}
{_{4}F_{3}} {\biggl( \begin{array}{cccc} \chi_5, & \chi_5^2, & \chi_5^3, & \chi_5^4 \\
\phantom{\chi_5} & \varepsilon, & \varepsilon, & \varepsilon \end{array}
\Big| \; 1 \biggr)}_{p}
- p
= c(p).
\end{equation*}
\end{theorem}

\noindent This result can be extended to almost all primes using ${_{n}G_{n}}[\cdots]$, as follows.
\begin{theorem}\cite[Theorem 1.4]{McC5}\label{cor_GStoMod}
If $p \neq 5$ is an odd prime is as defined in (\ref{for_ModForm}), then 
\begin{align*}
{_{4}G_{4}}
\biggl[ \begin{array}{cccc} \frac{1}{5}, & \frac{2}{5}, & \frac{3}{5}, & \frac{4}{5} \\[2pt]
 0, & 0, & 0, & 0 \end{array}
\Big| \; 1 \; \biggr]_p
-\bigl( \tfrac{5}{p} \bigr) \hspace{1pt} p
&=
c(p),
\end{align*}
where $\bigl( \tfrac{\cdot}{p} \bigr)$ is the Legendre symbol modulo $p$.
\end{theorem}
Results in \cite{M} establish congruences modulo $p^2$ between the classical hypergeometric series and the hypergeometric function over $\mathbb{F}_p$, for primes $p$ in certain congruence classes. 
In \cite{McC4} we extend these results to primes in additional congruence classes and,  in some cases to modulo $p^3$, using the predecessor to ${_{n}G_{n}}[\cdots]$.  

The main purpose of this paper is to extend to almost all primes the results in  \cite{L}, which relate the trace of Frobenius $a_p$ to a special value of a hypergeometric function over $\mathbb{F}_p$ when $p \equiv 1 \pmod {12}$.
In addition to their formal statement, the results in \cite{L} appear in various forms throughout that paper, all of of which are related by known transformations for hypergeometric function over finite fields.
We recall one such version of \cite[Theorem 2.1]{L}.
 \begin{theorem}[Lennon \cite{L} \S 2.2]\label{thm_Lennon}
Let $p \equiv 1 \pmod{12}$ be prime and let $\psi \in \widehat{\mathbb{F}^{*}_{p}}$ be a character of order 12. Consider an elliptic curve $E/\mathbb{F}_p$ of the form $E:y^2 = x^3+ax+b$ with $j(E)\neq 0, 1728$. Then
\begin{equation*}
a_p(E) = 
\psi^3 \left (-\frac{a^3}{27}\right) \cdot 
{_{2}F_{1}} {\biggl( \begin{array}{cc} \psi, & \psi^5 \\
 \phantom{\psi,} & \varepsilon \end{array}
\Big| \; \frac{4a^3+27b^2}{4a^3} \biggr)}_{p}.
\end{equation*}
\end{theorem}
\noindent Theorem \ref{thm_Lennon} generalizes \cite[Theorem 1.2]{F2} and other results from Fuselier's thesis \cite{F} which provide similar results for various families of elliptic curves.
In attempting to extend Theorem \ref{thm_Lennon} beyond $p \equiv 1 \pmod{12}$  one might consider using  
${_{2}G_{2}}
\biggl[ \begin{array}{cc} \frac{1}{12},  & \frac{5}{12}\vspace{.05in}\\[2pt]
 0,  & 0 \end{array}
\Big| \; \displaystyle\frac{4a^3}{4a^3+27b^2} \; \biggr]_p,$
as suggested by Lemma \ref{lem_G_to_F}. However this leads to poor results when $p \not\equiv 1 \pmod{12}$.
Results where ${_{n}G_{n}}[\cdots]$ extend those involving ${_{n+1}F_{n}}(\cdots)$, seem to work best when the arguments of  ${_{n}G_{n}}[\cdots]$ appear in sets such that for each denominator all possible relatively prime numerators are represented.    
This is reflected in Theorem \ref{thm_trace}.

Hypergeometric functions over finite fields have been applied to many areas but most interestingly perhaps has been their relationships to modular forms \cite{AO, E, F2, FOP, McC5, M, O, P} and their use in evaluating the number of points over $\mathbb{F}_{p}$ on certain algebraic varieties \cite{AO, F2, McC5, V}.
Lemma \ref{lem_G_to_F} allows these results to be expressed in terms of ${_{n}G_{n}}[\cdots]$ also.
Many of these cited results are based on ${_{n+1}F_{n}}(\cdots)$ with arguments which are characters of order $\leq 2$ and hold for all odd primes.
However there is much scope for developing results where the characters involved have higher orders, in which case, these functions will be defined for primes in certain congruence classes and ${_{n}G_{n}}[\cdots]$ allows the possibility to extend these results to a wider class of primes.


\section{Proofs of Theorem \ref{thm_trace} and Corollary \ref{cor_trace}}\label{sec_Proofs}

We first prove a preliminary result which we will require later for the proof of our main result.
\begin{lemma}\label{lem_pGamma}
Let $p$ be prime. For $0 \leq j \leq p-2$ and $t \in \mathbb{Z}^{+}$ with $p \nmid t$,
\begin{equation}\label{lem_gammaptj}
\gfp{\Big\langle {\tfrac{tj}{p-1}} \Big\rangle}\; {\omega(t^{tj}) \displaystyle\prod_{h=1}^{t-1} \gfp{\tfrac{h}{t}}} = \displaystyle\prod_{h=0}^{t-1} \gfp{\Big\langle \tfrac{h}{t} + \tfrac{j}{p-1} \Big\rangle}
\end{equation}
and
\begin{equation}\label{lem_gammamtj}
\gfp{\Big\langle {\tfrac{-tj}{p-1}} \Big\rangle}\; {\omega(t^{-tj}) \displaystyle\prod_{h=1}^{t-1} \gfp{\tfrac{h}{t}}} = \displaystyle\prod_{h=0}^{t-1} \gfp{\Big\langle \tfrac{1+h}{t} - \tfrac{j}{p-1} \Big\rangle}.
\end{equation}
\end{lemma}

\begin{proof}
Fix  $0 \leq j \leq p-2$ and let $k \in \mathbb{Z}_{\geq 0}$ be defined such that 
\begin{equation}\label{cond_kptj}
k \left(\tfrac{p-1}{t} \right) \leq j < (k+1) \left(\tfrac{p-1}{t} \right).
\end{equation}

\noindent Letting $m=t$ and $x=\frac{tj}{p-1} - k$ in (\ref{for_pGammaMult}) yields
\begin{equation}\label{for_ptj0}
\displaystyle\prod_{h=0}^{t-1}\gfp{\tfrac{j}{p-1} + \tfrac{h-k}{t}}= \omega \left(t^{(1-\frac{tj}{p-1} + k)(1-p)}\right) \gfp{\tfrac{tj}{p-1} - k} \prod_{h=1}^{t-1} \gfp{\tfrac{h}{t}}.
\end{equation}

\noindent We note that $0 \leq k < t$. Using (\ref{cond_kptj}) we see that if $0 \leq h < t$ then $0 \leq \tfrac{h-k}{t}+ \tfrac{j}{p-1} < 1.$
Therefore, if $1 \leq k < t$ then
\begin{align}\label{for_ptj2}
\notag \displaystyle\prod_{h=0}^{t-1} \gfp{\tfrac{h-k}{t}+ \tfrac{j}{p-1}} &=
\displaystyle\prod_{h=0}^{t-1} \gfp{\Big\langle\tfrac{h-k}{t}+ \tfrac{j}{p-1}\Big\rangle} \\
\notag &=
\displaystyle\prod_{h=0}^{k-1} \gfp{\Big\langle \tfrac{t+h-k}{t}+ \tfrac{j}{p-1} \Big\rangle}
\displaystyle\prod_{h=k}^{t-1} \gfp{\Big\langle \tfrac{h-k}{t}+ \tfrac{j}{p-1} \Big\rangle}\\
\notag &=
\displaystyle\prod_{h=t-k}^{t-1} \gfp{\Big\langle \tfrac{h}{t}+ \tfrac{j}{p-1} \Big\rangle}
\displaystyle\prod_{h=0}^{t-k-1} \gfp{ \Big\langle \tfrac{h}{t}+ \tfrac{j}{p-1} \Big\rangle}\\
&=
\displaystyle\prod_{h=0}^{t-1} \gfp{\Big\langle\tfrac{h}{t}+ \tfrac{j}{p-1} \Big\rangle}.
\end{align}
The result in (\ref{for_ptj2}) also holds when $k=0$. Substituting (\ref{for_ptj2}) into (\ref{for_ptj0}) and noting that $\gfp{\Big\langle {\tfrac{tj}{p-1}} \Big\rangle} = \gfp{\tfrac{tj}{p-1} -k}$, by (\ref {cond_kptj}), yields (\ref{lem_gammaptj}).

We use a similar argument to prove (\ref{lem_gammamtj}). The result is trivial for $j=0$. Fix  $0 < j \leq p-2$ and let $k \in \mathbb{Z}^{+}$ be defined such that 
\begin{equation}\label{cond_kmtj}
(k-1) \left(\tfrac{p-1}{t} \right) < j \leq k \left(\tfrac{p-1}{t} \right).
\end{equation}

\noindent Letting $m=t$ and $x=k - \frac{tj}{p-1} $ in (\ref{for_pGammaMult}) yields
\begin{equation}\label{for_mtj0}
\displaystyle\prod_{h=0}^{t-1} \gfp{\tfrac{k+h}{t} - \tfrac{tj}{p-1}} = \omega \left(t^{(1-k +\frac{tj}{p-1})(1-p)}\right) \gfp{k - \tfrac{tj}{p-1} } \prod_{h=1}^{t-1} \gfp{\tfrac{h}{t}}.
\end{equation}

\noindent We note that $1 \leq k \leq t$. Using (\ref{cond_kmtj}) we see that if $0 \leq h < t$ then $0 \leq \tfrac{k+h}{t}- \tfrac{j}{p-1} < 1$. 
Therefore, if  $1 < k \leq t$ then
\begin{align}\label{for_mtj2}
\notag \displaystyle\prod_{h=0}^{t-1} \gfp{\tfrac{k+h}{t} - \tfrac{j}{p-1}} &=
\displaystyle\prod_{h=0}^{t-1} \gfp{\Big\langle \tfrac{k+h}{t} - \tfrac{j}{p-1} \Big\rangle}\\
\notag &=
\displaystyle\prod_{h=0}^{t-k} \gfp{\Big\langle \tfrac{k+h}{t} - \tfrac{j}{p-1} \Big\rangle}
\displaystyle\prod_{h=t-k+1}^{t-1} \gfp{\Big\langle \tfrac{k+h-t}{t} - \tfrac{j}{p-1} \Big\rangle}\\
\notag &=
\displaystyle\prod_{h=k-1}^{t-1} \gfp{\Big\langle \tfrac{1+h}{t}- \tfrac{j}{p-1} \Big\rangle}
\displaystyle\prod_{h=0}^{k-2} \gfp{\Big\langle \tfrac{1+h}{t}- \tfrac{j}{p-1}\Big\rangle}\\
&=
\displaystyle\prod_{h=0}^{t-1} \gfp{\Big\langle\tfrac{1+h}{t}- \tfrac{j}{p-1} \Big\rangle}.
\end{align}
The result in (\ref{for_mtj2}) also holds when $k=1$. Substituting (\ref{for_mtj2}) into (\ref{for_mtj0}) and noting that $\gfp{\Big\langle {\tfrac{-tj}{p-1}} \Big\rangle} = \gfp{\tfrac{-tj}{p-1} +k}$, by (\ref{cond_kmtj}), yields (\ref{lem_gammamtj}).
\end{proof}


\begin{proof}[Proof of Theorem \ref{thm_trace}]
We note from the outset that  $a\neq 0, b\neq 0$ and $-\frac{27b^2}{4a^3}\neq1$ as $j(E) \neq 0,1728$. 
Initially the proof proceeds along similar lines to the proofs of \cite[Thm 1.2]{F2} and  \cite[Thm 2.1]{L} by using (\ref{for_CtgPts}) to evaluate $\# E(\mathbb{F}_p)$. 
However we then transfer to the $p$-adic setting using the Gross-Koblitz formula (Theorem \ref{thm_GrossKoblitz}) and use properties of the $p$-adic gamma function, including Lemma \ref{lem_pGamma}, to prove the desired result.
By (\ref{for_CtgPts}) we have that
\begin{align}\label{for_Np1}
\notag 
p ( \# E(\mathbb{F}_p) -1)
\notag 
&= p^2 +\sum_{y \in \mathbb{F}_p^*} \sum_{x_1, x_2 \in \mathbb{F}_p}
\theta(y \: (x_1^3 +ax_1+b-x_2^2))\\
\notag 
&= p^2 + \sum_{y \in \mathbb{F}_p^*} \theta(y b) + \sum_{y,  x_2 \in \mathbb{F}_p^*} \theta(yb- yx_2^2 ) + \sum_{y,  x_1 \in \mathbb{F}_p^*} \theta(y x_1^3+ayx_1+yb )\\
& \phantom{= p^2 \;} + \sum_{y, x_1, x_2 \in \mathbb{F}_p^*}\theta(yx_1^3 +ayx_1+by-yx_2^2)).
\end{align}
We now examine each sum of (\ref{for_Np1}) in turn and will refer to them as $S_1$ to $S_4$ respectively. 
Using (\ref{sum_AddChar}) we see that
\begin{equation*}
S_1 = \sum_{y \in \mathbb{F}_p^*} \theta(y b)  = -1.
\end{equation*}
We use (\ref{for_AddProp}) and (\ref{for_AddtoGauss}) to expand the remaining terms as expressions in Gauss sums. This exercise has also been carried out in the proof of \cite[Thm 2.1]{L} so we only give a brief account here.
Let $T$ be a fixed generator for the group of characters of $\mathbb{F}_p^*$. Then
\begin{align*}
S_2 &= \sum_{y,  x_2 \in \mathbb{F}_p^*} \theta(yb- yx_2^2 )  \\
&= \frac{1}{(p-1)^2}\sum_{r,s = 0}^{p-2} g(T^{-r}) \; g(T^{-s}) \; T^r(b)\; T^s(-1) \sum_{x_2 \in \mathbb{F}_p^*} T^{2s}(x_2) \sum_{y \in \mathbb{F}_p^*} T^{r+s}(y).
\end{align*}
We now apply (\ref{for_TOrthEl}) to the last summation on the right, which yields $(p-1)$ if $r=-s$ and zero otherwise. So
\begin{align*}
S_2 
&= \frac{1}{(p-1)}\sum_{s = 0}^{p-2}  g(T^{s}) \; g(T^{-s}) \; T^{-s}(b) \; T^s(-1) \sum_{x_2 \in \mathbb{F}_p^*} T^{2s}(x_2).
\end{align*}
Again we apply (\ref{for_TOrthEl}) to the last summation on the right, which yields $(p-1)$ if $s=0$ or $s=\tfrac{p-1}{2}$, and zero otherwise.
Thus, and using (\ref{for_GaussConj}), we get that
\begin{align*}
S_2 
&=  g(\varepsilon) \; g(\varepsilon) + \; g(\phi) \; g(\phi) \; \phi(-b) = 1 +p \, \phi(b).
\end{align*}
Similarly,
\begin{align*}
S_3 &= \sum_{y,  x_1 \in \mathbb{F}_p^*} \theta(y x_1^3+ayx_1+yb ) \\
&= \frac{1}{(p-1)^3}\sum_{r,s,t = 0}^{p-2}g(T^{-r}) \; g(T^{-s}) \; g(T^{-t})\; T^s(a)\; T^t(b) \\
& \qquad \qquad \qquad \qquad \qquad \qquad \qquad  \qquad \cdot 
\sum_{x_1 \in \mathbb{F}_p^*}T^{3r+s}(x_1) \sum_{y \in \mathbb{F}_p^*} T^{r+s+t}(y),
\end{align*}
and
\begin{align*}
S_4 &= \sum_{y, x_1, x_2 \in \mathbb{F}_p^*}\theta(yx_1^3 +ayx_1+by-yx_2^2))\\
&= \frac{1}{(p-1)^4}\sum_{j,r,s,t = 0}^{p-2}g(T^{-j}) \; g(T^{-r}) \; g(T^{-s}) \; g(T^{-t})\; T^r(a) \; T^s(b) \; T^t(-1) \\
& \qquad \qquad \qquad \qquad \qquad \qquad
\cdot \sum_{x_1 \in \mathbb{F}_p^*}T^{3j+r}(x_1) \sum_{y \in \mathbb{F}_p^*} T^{j+r+s+t}(y) \sum_{x_2 \in \mathbb{F}_p^*}T^{2t}(x_2) .
\end{align*}
We now apply (\ref{for_TOrthEl}) to the last summation on the right of $S_4$, which yields $(p-1)$ if $t=0$ or $t=\frac{p-1}{2}$ and zero otherwise.
In the case $t=0$ we find that
\begin{align*}
S_{4,t=0} 
&=-S_3.
\end{align*}
When $t=\frac{p-1}{2}$ we get, after applying  (\ref{for_TOrthEl}) twice more, 
 \begin{equation*}
\notag S_{4,t=\frac{p-1}{2}} 
= \frac{\phi(-b)}{(p-1)}\sum_{j = 0}^{p-2}g(T^{-j}) \; g(T^{\frac{p-1}{2}-2j}) \; g(T^{3j}) \; g(T^{\frac{p-1}{2}})\; T^{-3j}(a) \; T^{2j}(b)  .
\end{equation*}

\noindent Combining (\ref{def_ap}), (\ref{for_Np1}) and the evaluations of $S_1, S_2, S_3$ and $S_4$ we find that
\begin{equation}\label{for_ap1}
a_p(E)=  -\frac{\phi(b)\, p}{(p-1)}  -  \frac{\phi(-b)}{p(p-1)}\sum_{j = 1}^{p-2}g(T^{-j}) \; g(T^{\frac{p-1}{2}-2j}) \; g(T^{3j}) \; g(T^{\frac{p-1}{2}})\; T^{j}(\tfrac{b^2}{a^3})  .
\end{equation}
We know from Theorem \ref{thm_HD} with $\chi=\phi=T^{\frac{p-1}{2}}$ and $\psi=T^{-2j}$ that
\begin{equation}\label{for_HD1}
 g(T^{\frac{p-1}{2}-2j}) = \frac{g(T^{-4j}) \; g(T^{\frac{p-1}{2}}) \; T^{4j}(2)}{g(T^{-2j})}.
\end{equation}
Accounting for (\ref{for_HD1}) in (\ref{for_ap1}) and applying (\ref{for_GaussConj}) with $\chi=\phi=T^{\frac{p-1}{2}}$ gives us
\begin{equation}\label{for_ap2}
a_p(E)=  \frac{-\phi(b)\, p}{(p-1)} \left[ 1 + \frac{1}{p}\sum_{j = 1}^{p-2} \frac{g(T^{-j}) \; g(T^{3j}) \; g(T^{-4j})}{g(T^{-2j})}\; T^{j}(\tfrac{16b^2}{a^3})  \right].
\end{equation}
We now take $T$ to be the inverse of the Teichm\"{u}ller character, i.e., $T= \bar{\omega}$, and 
use the Gross-Koblitz formula (Theorem \ref{thm_GrossKoblitz}) to convert (\ref{for_ap2}) to an expression involving the $p$-adic gamma function. This yields
\begin{multline}\label{for_ap3}
a_p(E)=  \frac{-\phi(b)\, p}{(p-1)} \left[ 1 - \sum_{j = 1}^{p-2}
(-p)^{\left(  \lfloor{\frac{-2j}{p-1}}\rfloor  -  \lfloor{\frac{-j}{p-1}}\rfloor -  \lfloor{\frac{3j}{p-1}}\rfloor -  \lfloor{\frac{-4j}{p-1}}\rfloor  -1 \right)}
\right. \\ \left. \cdot
\frac{\biggfp{ \langle \frac{-j}{p-1}\rangle}\biggfp{\langle \frac{3j}{p-1}\rangle} \biggfp{\langle \frac{-4j}{p-1}\rangle}}{\biggfp{\langle \frac{-2j}{p-1}\rangle}}
\bar{\omega}^j(\tfrac{16b^2}{a^3}) 
\right].
\end{multline}
Next we use Lemma \ref{lem_pGamma} to transform the components of (\ref{for_ap3}) which involve the $p$-adic gamma function.
After some tidying up we then get
\begin{multline*}\label{for_ap4}
a_p(E)=  \frac{-\phi(b)\, p}{(p-1)} \left[ 1 - \sum_{j = 1}^{p-2}
(-p)^{\left(  \lfloor{\frac{-2j}{p-1}}\rfloor  -  \lfloor{\frac{-j}{p-1}}\rfloor -  \lfloor{\frac{3j}{p-1}}\rfloor -  \lfloor{\frac{-4j}{p-1}}\rfloor  -1 \right)}\;
\biggfp{ 1- \tfrac{j}{p-1}}
\biggfp{ \tfrac{j}{p-1}}
\right. \\ \cdot \left.
\frac{
\biggfp{\langle \frac{1}{4}- \frac{j}{p-1}\rangle}
\biggfp{\langle \frac{3}{4}- \frac{j}{p-1}\rangle}
\biggfp{\langle \frac{1}{3}+ \frac{j}{p-1}\rangle}
\biggfp{\langle \frac{2}{3}+ \frac{j}{p-1}\rangle}
}
{
\biggfp{\frac{1}{4}}
\biggfp{\frac{3}{4}}
\biggfp{\frac{1}{3}}
\biggfp{\frac{2}{3}}
}
\; \bar{\omega}^j(\tfrac{27b^2}{4a^3}) 
\right].
\end{multline*}
We note for $0 \leq j \leq p-2$ that 
$$ \lfloor{\tfrac{-4j}{p-1}}\rfloor - \lfloor{\tfrac{-2j}{p-1}}\rfloor =  \lfloor{\tfrac{1}{4} -\tfrac{j}{p-1}}\rfloor + \lfloor{\tfrac{3}{4}-\tfrac{j}{p-1}}\rfloor,$$
and when $1 \leq  j \leq p-2$ that 
$$\lfloor{\tfrac{-j}{p-1}}\rfloor + \lfloor{\tfrac{3j}{p-1}}\rfloor + 1 = \lfloor{\tfrac{1}{3} + \tfrac{j}{p-1}}\rfloor + \lfloor{\tfrac{2}{3}+\tfrac{j}{p-1}}\rfloor .$$
Also, by (\ref{for_pGammaOneMinus}) we have that, for  $0 \leq j \leq p-1$,
$$ \biggfp{ 1- \tfrac{j}{p-1}}\biggfp{ \tfrac{j}{p-1}} = (-1)^{p-j}= (-1)^p \; \bar{\omega}^j(-1).$$
Therefore
\begin{align*}
a_p(E)
&=  \frac{-\phi(b)\, p}{(p-1)} \left[  \sum_{j = 0}^{p-2}
(-p)^{\left( - \lfloor{\tfrac{1}{4} -\tfrac{j}{p-1}}\rfloor - \lfloor{\tfrac{3}{4}-\tfrac{j}{p-1}}\rfloor - \lfloor{\tfrac{1}{3} + \tfrac{j}{p-1}}\rfloor - \lfloor{\tfrac{2}{3}+\tfrac{j}{p-1}}\rfloor \right)}
\right. \\ &  \cdot \left.
\frac{\biggfp{\langle \frac{1}{4}- \frac{j}{p-1}\rangle}\biggfp{\langle \frac{3}{4}- \frac{j}{p-1}\rangle}}{\biggfp{\frac{1}{4}}\biggfp{\frac{3}{4}}}
\cdot \frac{\biggfp{\langle -\frac{2}{3}+ \frac{j}{p-1}\rangle}\biggfp{\langle -\frac{1}{3}+ \frac{j}{p-1}\rangle}}{\biggfp{\langle -\frac{2}{3}\rangle}\biggfp{\langle -\frac{1}{3}\rangle}}
\; \bar{\omega}^j(-\tfrac{27b^2}{4a^3}) 
\right]\\
&= 
\phi(b) \cdot p \cdot
{_{2}G_{2}}
\biggl[ \begin{array}{cc} \frac{1}{4},  & \frac{3}{4}\vspace{.05in}\\
 \frac{1}{3},  & \frac{2}{3} \end{array}
\Big| \; {-\frac{27b^2}{4a^3}} \; \biggr]_p.
\end{align*}
\end{proof}


\begin{remark}
Using (\ref{for_CtgPts}) to evaluate the number of  points on certain algebraic varieties over finite fields is by no means new. 
However, the author first observed the technique in the work of Fuselier \cite{F, F2} where it was used to relate these evaluations to hypergeometric functions over finite fields. 
These methods were subsequently used by Lennon \cite{L} in generalizing Fuselier’s work and, as we've seen, also form part of our proof of  Theorem \ref{thm_trace}. 
\end{remark}


\begin{proof}[Proof of Corollary \ref{cor_trace}]
As noted in the introduction, when $p>3$ then any elliptic curve $E/\mathbb{F}_p$ is isomorphic to an elliptic curve of the form $E^{\prime}:y^2 = x^3+ax+b$. Therefore $a_p(E)=a_p(E^{\prime})$ and Theorem \ref{thm_trace} can be used to evaluate $a_p(E)$.
We also note that $j(E)=j(E^{\prime}) = \frac{1728 \cdot 4a^3}{4a^3 + 27b^2}$ and so $ 1- \frac{1728}{j(E)}=-\frac{27b^2}{4a^3} $. As $E$ and $E^{\prime}$ are related by an admissible change of variables, this implies $c_6(E) = c_6(E^{\prime}) \cdot u^6$ for some $u \in \mathbb{F}_p^{*}$. 
Now $c_6(E^{\prime}) = -27\cdot 32\cdot b$ so $\phi(b) =\phi(-6\cdot c_6(E))$ as required. 
\end{proof}


\section{Concluding Remarks}\label{sec_cr}
\subsection{The $p=3$ Case}
Theorem \ref{thm_trace} considers elliptic curves over $\mathbb{F}_p$ for primes $p>3$. 
While ${_{n}G_{n}}[\cdots]_p$ is not defined for $p=2$ it is defined for $p=3$ once the parameters are $3$-adic integers.
As the parameters of the ${_{2}G_{2}}[\cdots]_p$ in Theorem \ref{thm_trace} are not all $3$-adic integers it is clear that the result cannot be extended to $p=3$ using the same function.
However we can say something about the $p=3$ case.
Any elliptic curve over $\mathbb{F}_3$, whose $j$-invariant is non-zero, is isomorphic to a curve of the form $E: y^2=x^3+ax^2+b$ with both $a$ and $b$ non-zero \cite[App. A]{Si}.
It is an easy exercise to evaluate $a_3(E)$ and to show that
\begin{equation*}
a_3(E) = \phi(a)  \cdot
{_{2}G_{2}}
\biggl[ \begin{array}{cc} 0,  & 0 \vspace{.05in}\\
0,  & \frac{1}{2} \end{array}
\Big| \; {-\frac{a}{b}} \; \biggr]_3. 
 \end{equation*}
This relationship is somewhat contrived however and direct calculation of $a_3(E)$ is much more straightforward.

\subsection{Transformation Properties of ${_{n}G_{n}}[\cdots]_p$}
As mentioned in Section \ref{sec_Properties}, hypergeometric functions over finite fields were originally defined by Greene \cite{G} as analogues of classical hypergeometric functions.
His motivation was to develop the area of character sums and their evaluations through parallels with the classical functions, and, in particular, with their transformation properties. 
His endeavor was largely successful and analogues of various classical transformations were found \cite{G}.
Some others were recently provided by the author in \cite{McC6}.
These transformations for hypergeometric functions over finite fields can obviously be re-written in terms of  ${_{n}G_{n}}[\cdots]_p$ via Lemma  \ref{lem_G_to_F} and these results will hold for all $p$ where the original characters existed over $\mathbb{F}_p$.
It is an interesting question to consider if these transformations can then be extended to almost all $p$ and become transformations for ${_{n}G_{n}}[\cdots]_p$ in full generality.
This is something yet to be considered and may be the subject of forthcoming work.

\subsection{$q$-Version of ${_{n}G_{n}}[\cdots]_p$}
As discussed in Section \ref{sec_Properties}, ${_{n}G_{n}}[\cdots]_p$ extends hypergeometric functions over finite fields, as defined in Definition \ref{def_F}, to the $p$-adic setting.
Definition \ref{def_F} can easily be extended to $\mathbb{F}_q$ where $q$ is a prime power and indeed, this is how it was originally defined in \cite[Definition 1.4]{McC6}. 
In a similar manner to the proof of Lemma \ref{lem_G_to_F}, we can then use the Gross-Koblitz formula (not as quoted in Theorem \ref{thm_GrossKoblitz} but its $\mathbb{F}_q$-version) to transform the  hypergeometric function over $\mathbb{F}_q$ to an expression involving products of the $p$-adic gamma function.
Generalizing the resulting expression  yields the following $q$-version of ${_{n}G_{n}}[\cdots]_p$.
We now let $\omega$ denote the Teichm\"{u}ller character of $\mathbb{F}_q$.
\begin{defi}\label{def_Gq}
Let $q=p^r$, for $p$ an odd prime and $r \in \mathbb{Z}^{+}$, and let $t \in \mathbb{F}_q$. 
For $n \in \mathbb{Z}^{+}$ and $1 \leq i \leq n$, let $a_i, b_i \in \mathbb{Q} \cap \mathbb{Z}_p$.
Then we define  
\begin{multline*}
{_{n}G_{n}}
\biggl[ \begin{array}{cccc} a_1, & a_2, & \dotsc, & a_n \\
 b_1, & b_2, & \dotsc, & b_n \end{array}
\Big| \; t \; \biggr]_q
: = \frac{-1}{q-1}  \sum_{j=0}^{q-2} 
(-1)^{jn}\;
\bar{\omega}^j(t)\\
\times \prod_{i=1}^{n} 
 \prod_{k=0}^{r-1} 
\frac{\biggfp{\langle (a_i -\frac{j}{q-1} )p^k \rangle}}{\biggfp{\langle a_i p^k \rangle}}
\frac{\biggfp{\langle (-b_i +\frac{j}{q-1}) p^k \rangle}}{\biggfp{\langle -b_i p^k\rangle}}
(-p)^{-\lfloor{\langle a_i p^k \rangle -\frac{j p^k}{q-1}}\rfloor -\lfloor{\langle -b_i p^k\rangle +\frac{j p^k}{q-1}}\rfloor}.
\end{multline*}
\end{defi}
When $q=p$ in Definition \ref{def_Gq} we recover ${_{n}G_{n}}[\cdots]_p$ as per Definition \ref{def_Gp}.
We believe ${_{n}G_{n}}[\cdots]_q$ could be used to generalize results involving hypergeometric functions over $\mathbb{F}_q$ which are restricted to $q$ in certain congruence classes (e.g those in \cite{L}).
However we do not examine this here for the following reason.
The main purpose of this paper is to demonstrate that  ${_{n}G_{n}}[\cdots]_p$ can be used to extend results involving hypergeometric functions over $\mathbb{F}_p$, which are limited to primes in certain congruence classes, and thus avoid the need to work over $\mathbb{F}_q$. 


\section{Acknowledgments}\label{sec_ack}
I would like to thank the referee for some helpful comments and suggestions to improve this paper.


\vspace{12pt}

\begin{thebibliography}{999}

\bibitem{AO} S. Ahlgren, K. Ono, \emph{A Gaussian hypergeometric series evaluation and Ap{\'e}ry number congruences}, J. Reine Angew. Math. \textbf{518} (2000), 187--212.

\bibitem{BEW} B. Berndt, R. Evans, K. Williams, \emph{Gauss and Jacobi Sums}, Canadian Mathematical Society Series of Monographs and Advanced Texts, A Wiley-Interscience Publication, John Wiley \& Sons, Inc., New York, 1998.

\bibitem{E} R. Evans, \emph{Hypergeometric $_3F_2(1/4)$ evaluations over finite fields and Hecke eigenforms}, Proc. Amer. Math. Soc. \textbf{138} (2010), 517--531.

 \bibitem{F} J. Fuselier, \emph{Hypergeometric functions over finite fields and relations to modular forms and elliptic curves}, Ph.D. thesis, Texas A$\&$M University, 2007. 

\bibitem{F2} J. Fuselier, \emph{Hypergeometric functions over $\mathbb{F}_p$ and relations to elliptic curves and modular forms}, Proc. Amer. Math. Soc. \textbf{138} (2010), no.1, 109--123.

\bibitem{FOP} S. Frechette, K. Ono, and M. Papanikolas, \emph{Gaussian hypergeometric functions and traces of Hecke operators}, Int. Math. Res. Not. \textbf{2004}, no. 60, 3233--3262.

\bibitem{G} J. Greene, \emph{Hypergeometric functions over finite fields}, Trans. Amer. Math. Soc. \textbf{301} (1987), 77--101.

 \bibitem{GK} B. Gross, N. Koblitz, \emph{Gauss sums and the p-adic $\Gamma$-function}, Ann. Math. \textbf{109} (1979), no. 3, 569--581.


\bibitem{K} N. M. Katz, \emph{Exponential Sums and Differential Equations}, Princeton Univ. Press, Princeton, 1990.

\bibitem{Kn} A.W. Knapp, \emph{Elliptic Curves}, Mathematical Notes, 40, Princeton University Press, Princeton, New Jersey, 1992.

\bibitem{Ko} N. Koblitz, \emph{$p$-adic analysis: a short course on recent work}, London Math. Soc. Lecture Note Series, \textbf{46}. Cambridge University Press, Cambridge-New York, 1980.

\bibitem{L} C. Lennon, \emph{Gaussian hypergeometric evaluations of traces of Frobenius for elliptic curves}, Proc. Amer. Math. Soc. \textbf{139} (2011), no. 6, 1931--1938. 

\bibitem{McC4} D. McCarthy, \emph{Extending Gaussian hypergeometric series to the $p$-adic setting}, Int. J. Number Theory \textbf{8} (2012), no. 7, 1581--1612.

\bibitem{McC5} D. McCarthy, \emph{On a supercongruence conjecture of Rodriguez-Villegas}, Proc. Amer. Math. Soc. \textbf{140} (2012), 2241--2254. 

\bibitem{McC6} D. McCarthy, \emph{Transformations of well-poised hypergeometric functions over finite fields}, Finite Fields and Their Applications, accepted for publication, arXiv:1204.4377.

\bibitem{M} E. Mortenson, \emph{Supercongruences for truncated ${}_{n+1}F_{n}$ hypergeometric series with applications to certain weight three newforms}, Proc. Amer. Math. Soc. \textbf{133} (2005), no. 2, 321--330.

\bibitem{O} K. Ono, \emph{Values of Gaussian hypergeometric series}, Trans. Amer. Math. Soc. \textbf{350} (1998), 1205--1223.

\bibitem{P} M. Papanikolas, \emph{A formula and a congruence for Ramanujan's $\tau$-function}, Proc. Amer. Math. Soc. \textbf{134} (2006), 333--341.

\bibitem{Si} J. H. Silverman, \emph{The Arithmetic of Elliptic Curves}, 2nd ed.,  Springer-Verlag, New York, 2009.

\bibitem{V} V. Vega, \emph{Relations between hypergeometric functions over finite fields and algebraic curves}, Int. J. Number Theory \textbf{7} (2011), no. 8, 2171--2195.

\end{thebibliography}
\end{document}